%% file: pm.tex
\documentclass[11 pt]{article}

\usepackage{fullpage}
\usepackage{color}
\usepackage{dsfont}
\usepackage{ifpdf}
\usepackage{hyperref}
\usepackage{dcolumn}
\usepackage{url}
\usepackage{amsmath}
\usepackage{amscd}
\usepackage{amsfonts}
\usepackage{amssymb}
\usepackage{bm}   
\usepackage{bbm}   
\usepackage{verbatim}
\usepackage{stmaryrd}
\usepackage{amsthm}
\usepackage{graphicx}
\usepackage{cite}
\usepackage{authblk}


\input{shortcuts}

\begin{document}

\title{Exponential Bounds for Convergence of Entropy Rate Approximations in 
Hidden Markov Models Satisfying a Path-Mergeability Condition}

\author{Nicholas F. Travers \thanks{Department of Mathematics, Technion--Israel Institute of Technology.  
E-mail - \texttt{travers@tx.technion.ac.il}. (This work was completed primarily while the author was at the 
University of California, Davis.)  }}

\date{} 

\maketitle

\vspace{- 10 mm}
\begin{abstract}
A hidden Markov model (HMM) is said to have \emph{path-mergeable states} if for any two states $i, j$ 
there exists a word $w$ and state $k$ such that it is possible to transition from both $i$ and $j$
to $k$ while emitting $w$.  We show that for a finite HMM with path-mergeable states the block 
estimates of the entropy rate converge exponentially fast. We also show that the path-mergeability
property is asymptotically typical in the space of HMM topologies and easily testable.
\end{abstract}
\vspace{-5 mm}

\section{Introduction}
\label{sec:Introduction}

Hidden Markov models (HMMs) are generalizations of Markov chains in which the underlying Markov state 
sequence $(S_t)$ is observed through a noisy or lossy channel, leading to a (typically) non-Markovian output 
process $(X_t)$. They were first introduced in the 50s as abstract mathematical models \cite{Gilb59a, Blac57a, Blac57b},
but have since proved useful in a number of concrete applications, such as speech recognition 
\cite{Juan91a, Rabi89a, Bahl83a, Jeli76a} and bioinformatics \cite{Siep04a, Eddy98a, Karp98a, Eddy95a, Bald94a}.

One of the earliest major questions in the study of HMMs \cite{Blac57b} was to determine the entropy rate of the output process
\begin{align*}
h &= \lim_{t \to \infty} H(X_t|X_1,...,X_{t-1}).
\end{align*}
Unlike for Markov chains, this actually turns out to be quite difficult for HMMs. Even in the finite case no general closed 
form expression is known, and it is widely believed that no such formula exists. A nice integral expression was provided 
in \cite{Blac57b}, but it is with respect to an invariant density that is not directly computable. 

In practice, the entropy rate $h$ is instead often estimated directly by the finite length block estimates
\begin{align*}
h(t) & = H(X_t|X_1,...,X_{t-1}).
\end{align*}
Thus, it is important to know about the rate of convergence of these estimates to ensure the quality of the approximation. 

Moreover, even in cases where the entropy rate can be calculated exactly, such as for unifilar HMMs, the rate of convergence
of the block estimates is still of independent interest. It is important for numerical estimation of various complexity measures,
such as the excess entropy and transient information \cite{Crut01a}, and is critical for an observer wishing to make predictions 
of future output from a finite observation sequence $X_1, ..., X_t$. It is also closely related to the rate of memory loss in the initial condition for 
HMMs, a problem that has been studied extensively in the field of filtering theory \cite{Sowe92a, Atar95a, Atar97a, Glan00b, Chig04a, 
Douc09a, Coll09a}. Though, primarily in the case of continuous real-valued outputs with gaussian noise or similar, rather than the 
discrete case we study here. 

No general bounds are known for the rate of convergence of the estimates $h(t)$, but exponential bounds have been established
for finite HMMs (with both a finite internal state set $\SM$ and finite output alphabet $\XM$) under various positivity assumptions. 
The earliest known, and mostly commonly cited, bound is given in \cite{Birc62a}, for finite functional HMMs with strictly positive transition 
probabilities in the underlying Markov chain. A somewhat improved bound under the same hypothesis is also given in \cite{Hoch99a}. 
Similarly, exponential convergence for finite HMMs with strictly positive symbol emission probabilities and an aperiodic underlying Markov 
chain is established in \cite{Pfis03a} (both for state-emitting and edge-emitting HMMs) using results from \cite{Glan00b}. 

Without any positivity assumptions, though, things become substantially more difficult. In the particular case of unifilar HMMs, we have
demonstrated exponential convergence in our recent work on synchronization with Crutchfield \cite{Trav11a, Trav11b}. Also, in reference
\cite{Han06} exponential convergence is established under some fairly technical hypotheses in studying entropy rate analyticity. But, no
general bounds on the convergence rate of the block estimates $h(t)$ have been demonstrated for all finite HMMs. 

Here we prove exponential convergence for finite HMMs (both state-emitting and edge-emitting) satisfying the following simple 
path-mergeability condition: \emph{For each pair of distinct states $i, j$ there exists a word $w$ and state $k$ such that it is possible 
to transition from both $i$ and $j$ to $k$ while emitting $w$.} We also show that this condition is easily testable (in computational time polynomial 
in the number of states and symbols) and asymptotically typical in the space of HMM topologies, in that a randomly selected topology will 
satisfy this condition with probability approaching 1, as the number of states goes to infinity. By contrast, the positivity conditions assumed 
in \cite{Birc62a, Hoch99a, Pfis03a}, as well as the unifilarity hypothesis we assume in \cite{Trav11a, Trav11b}, are satisfied only for a vanishingly
small fraction of HMM topologies in the limit that the number of states becomes large, and the conditions assumed in \cite{Han06} are, to our 
knowledge, not easily testable in an automated fashion like the path-mergeability condition. The conditions assumed in \cite{Han06}
are also intuitively somewhat stronger than path-mergeability in that they require the output process $(X_t)$ to have full support and require 
uniform exponential convergence of conditional probabilities\footnotemark{}, neither of which are required for path-mergeable HMMs. 

\footnotetext{E.g., there exist constants $K > 0$ and $0 < \alpha <1$ such that for any symbol $x$ and symbol sequence
$x_{-t -n}, ..., x_{-1}$, $t,n \in \N$, $|\P(X_0 = x| X_{-1} = x_{-1}, ... , X_{-t} = x_{-t}) - \P(X_0 = x| X_{-1} = x_{-1}, ..., X_{-t-n} = x_{-t-n})|
\leq K \alpha^t$.}

The structure of the remainder of the paper is as follows. In Section \ref{sec:DefinitionsAndNotation} we introduce the formal framework 
for our results, including more complete definitions for hidden Markov models and their various properties, as well as the entropy rate 
and its finite-block estimates. In Section \ref{sec:ResultsForEdgeEmittingHMMs} we provide proofs of our exponential convergence results 
for edge-emitting HMMs satisfying the path-mergeability property. In Section \ref{sec:RelationToStateEmittingHMMs} we use the edge-emitting
results to establish analogous results for state-emitting HMMs. Finally, in Section \ref{sec:TypicalityTestabilityOfPathMergeability} we 
provide the algorithmic test for path-mergeability and demonstrate that this property is asymptotically typical. The proof methods used
in Section \ref{sec:ResultsForEdgeEmittingHMMs} are based on the original coupling argument used in \cite{Birc62a}, but are substantially
more involved because of our weaker assumption.

\section{Definitions and Notation}
\label{sec:DefinitionsAndNotation}

By an \emph{alphabet} $\XM$ we mean simply a set of symbols, and by a \emph{word} $w$ over the alphabet $\XM$ we mean 
a finite string $w = x_1,...,x_n$ consisting of symbols $x_i \in \XM$. The \emph{length} of a word $w$ is the number of symbols it contains 
and is denoted by $|w|$. $\XM^*$ denotes the set of all (finite) words over an alphabet $\XM$, including the empty word $\lambda$. 

For a sequence $(a_n)$ (of symbols, random variables, real numbers, ... etc.) and integers $n \leq m$, $a_n^m$ denotes the 
finite subsequence $a_n, a_{n+1}, ... , a_m$. This notation is also extended in the natural way to the case $m = \infty$ or 
$n = - \infty$. In the case $n > m$, $a_n^m$ is interpreted as the null sequence or empty word. 

\subsection{The Entropy Rate and Finite-Block Estimates}
\label{sec:TheEntropyRateAndFiniteBlockEstimates}

Throughout this section, and the remainder of the paper, we adopt the following standard information theoretic 
conventions for logarithms (of any base):
\begin{align*}
0 \cdot \log (0) & \equiv \lim_{\xi \to 0^+} \xi \cdot \log (\xi)  = 0. \\
0 \cdot \log (1/0) & \equiv \lim_{\xi \to 0^+} \xi \cdot \log (1/\xi) = 0.
\end{align*}
Note that with these conventions the functions $\xi \log(\xi)$ and $\xi \log(1/\xi)$ are both continuous on $[0,1]$. 

\begin{Def}
The \emph{entropy} $H(X)$ of a discrete random variable $X$ is
\begin{align*}
H(X) \equiv - \sum_{x \in \XM} \P(x) \log_2 \P(x)
\end{align*}
where $\XM$ is the alphabet (i.e. set of possible values) of the random variable $X$ and  $\P(x) = \P(X=x)$. 
\end{Def}

\begin{Def}
For discrete random variables $X$ and $Y$ the \emph{conditional entropy} $H(X|Y)$ is
\begin{align*}
H(X|Y) & \equiv \sum_{y \in \YM} \P(y) \cdot H(X|Y=y) \\
& = - \sum_{y \in \YM} \P(y) \sum_{x \in \XM} \P(x|y) \log_2 \P(x|y)
\end{align*}
where $\XM$ and $\YM$ are, respectively, the alphabets of $X$ and $Y$, $\P(y) = \P(Y=y)$, and $\P(x|y) = \P(X=x|Y=y)$.
\end{Def}

Intuitively, the entropy $H(X)$ is the amount of uncertainty in predicting $X$, or equivalently, the amount of 
information obtained by observing $X$. The conditional entropy $H(X|Y)$ is the average uncertainty in
predicting $X$ given the observation of $Y$. These quantities satisfy the relations
\begin{align*}
0 \leq H(X|Y) \leq H(X) \leq \log_2|\XM|.
\end{align*}

\begin{Def}
Let $(X_t)$ be a discrete time stationary process over a finite alphabet $\XM$. The \emph{entropy rate} $h$ of
the process $(X_t)$ is the asymptotic per symbol entropy:
\begin{align}
\label{eq:hdef}
h \equiv \lim_{t \to \infty} H(X_1^t)/t 
\end{align}
where $X_1^t = X_1,...,X_t$ is interpreted as a single discrete random variable taking values in the cross product alphabet $\XM^t$. 
\end{Def}
Using stationarity it may be shown that this limit $h$ always exists and is approached monotonically from above. Further, it may 
be shown that the entropy rate may also be expressed as the monotonic limit of the conditional next symbol entropies $h(t)$.
That is, $h(t) \searrow h$, where 
\begin{align}
\label{eq:htdef}
h(t) \equiv H(X_t|X_1^{t-1}).
\end{align}

The non-conditioned estimates $H(X_1^t)/t$ can approach no faster than a rate of $1/t$. However, the conditional 
estimates $h(t) = H(X_t|X_1^{t-1})$ can approach much more quickly, and are therefore generally more useful. 
Our primary goal is to establish an exponential bound on the rate of convergence of these conditional finite-block estimates 
$h(t)$ for the output process of a HMM satisfying the path-mergeability property defined in Section \ref{sec:PathMergeability}. 

\subsection{Hidden Markov Models}
\label{sec:HiddenMarkovModels}

We will consider here only finite HMMs, meaning that both the internal state set $\SM$ and output alphabet $\XM$
are finite. There are two primary types: \emph{state-emitting} and \emph{edge-emitting}. The state-emitting variety 
is the simpler of the two, and also the more commonly studied, so we introduce them first. However, our primary 
focus will be on edge-emitting HMMs because the path-mergeability condition we study, as well as the block model 
presentation of Section \ref{sec:BlockModels} used in the proofs, are both more natural in this context. 
\begin{Def}
A \emph{state-emitting hidden Markov model} is a 4-tuple $(\SM, \XM, \TM, \OM)$ where:
\begin{itemize}
\item $\SM$ is a finite set of states. 
\item $\XM$ is a finite alphabet of output symbols.
\item $\TM$ is an $|\SM| \times |\SM|$ stochastic state transition matrix: $\TM_{ij} = \P(S_{t+1} = j | S_{t} = i)$.
\item $\OM$ is an $|\SM| \times |\XM|$ stochastic observation matrix: $\OM_{ix} = \P(X_t = x | S_t = i)$. 
\end{itemize}
\end{Def}

The state sequence $(S_t)$ for a state-emitting HMM is generated according to the Markov transition matrix $\TM$,
and the observed sequence $(X_t)$ has conditional distribution defined by the observation matrix $\OM$:
\begin{align*}
\P(X_n^m = x_n^m | S_{0}^{\infty} = s_{0}^{\infty}) 
= \P(X_n^m = x_n^m | S_n^m = s_n^m)
= \prod_{t = n}^{m} \OM_{s_t x_t}.
\end{align*}  

An important special case is when the observation matrix is deterministic, and the symbol $X_t$ is simply
a function of the state $S_t$. This type of HMMs, known as \emph{functional HMMs} or 
\emph{functions of Markov chains}, are perhaps the most simple variety conceptually, and also were the
first type to be heavily studied. The integral expression for the entropy rate provided in \cite{Blac57b} 
and exponential bound on convergence of the block estimates $h(t)$ established in \cite{Birc62a} 
both dealt with HMMs of this type.  

Edge-emitting HMMs are an alternative representation in which the symbol $X_t$ depends not simply 
on the current state $S_t$ but also the next state $S_{t+1}$, or rather the transition between them. 

\begin{Def}
An \emph{edge-emitting hidden Markov model} is a 3-tuple $(\SM, \XM, \{\TM^{(x)}\})$ where:
\begin{itemize}
\item $\SM$ is a finite set of states. 
\item $\XM$ is a finite alphabet of output symbols.
\item $\TM^{(x)}, x \in \XM$, are $|\SM| \times |\SM|$ sub-stochastic symbol-labeled transition matrices 
whose sum $\TM$ is stochastic. $\TM^{(x)}_{ij}$ is the probability of transitioning from $i$ to $j$ on symbol $x$.
\end{itemize}
\end{Def}

Visually, one can depict an edge-emitting HMM as a directed graph with labeled edges. The vertices are the states,
and for each $i,j,x$ with $\TM^{(x)}_{ij} > 0$ there is a directed edge from $i$ to $j$ labeled with the transition probability 
$p = \TM^{(x)}_{ij}$ and symbol $x$. The sum of the probabilities on all outgoing edges from each state is 1.

The operation of the HMM is as follows: From the current state $S_t$ the HMM picks an outgoing 
edge $E_t$ according to their probabilities, generates the symbol $X_t $ labeling this edge, and 
then follows the edge to the next state $S_{t+1}$. Thus, we have the conditional measure
\begin{align*}
\P( S_{t+1} = j, X_t = x |S_t = i, S_0^{t-1} = s_0^{t-1}, X_0^{t-1} = x_0^{t-1} )
= \P(S_{t+1} = j, X_t = x | S_t = i) 
= \TM^{(x)}_{ij}
\end{align*} 
for any $i \in \SM$, $t \geq 0$, and possible length-$t$ joint past $(s_0^{t-1},x_0^{t-1})$ which
may precede state $i$. From this it follows, of course, that the state sequence $(S_t)$ is a 
Markov chain with transition matrix $\TM = \sum_x \TM^{(x)}$. As for state-emitting HMMs, however,
we will be interested primarily in the observable output sequence $(X_t)$ rather than the internal
state sequence $(S_t)$, which is assumed to be hidden form the observer. 

\begin{Rem}
It is assumed that, for a HMM of any type, each symbol $x \in \XM$ may be actually be generated with positive 
probability. That is, for each $x \in \XM$, there exists $i \in \SM$ such that $\P(X_0 = x|S_0 = i) > 0$. Otherwise, the symbol 
$x$ is useless and the alphabet can be restricted to $\XM/\{x\}$. 
\end{Rem}

\subsubsection{Irreducibility and Stationary Measures}
\label{sec:IrreducibilityAndStationaryMeasures}

A HMM, either state-emitting or edge-emitting, is said to be \emph{irreducible} if the underlying Markov chain over states with 
transition matrix $\TM$ is irreducible. In this case, there exists a unique \emph{stationary distribution} $\pi$ over the states 
satisfying $\pi = \pi \TM$, and the joint state-symbol sequence $(S_t,X_t)_{t \geq 0}$ with initial state $S_0$ drawn according
to $\pi$ is itself a stationary process. We will henceforth assume all HMMs are irreducible and denote by $\P$ the (unique) 
stationary measure on joint state-symbol sequences satisfying $S_0 \sim \pi$. 

In the following, this measure $\P$ will be our primary focus. Unless otherwise specified, all random variables are assumed 
to be generated according to the stationary measure $\P$. In particular, the \emph{entropy rate} $h$ and \emph{block estimate $h(t)$} 
for a HMM $M$ are defined by (\ref{eq:hdef}) and (\ref{eq:htdef}) where $(X_t)$ is the stationary output process of $M$ with law $\P$.  

At times, however, it will be necessary to consider alternative measures in the proofs given by choosing the initial state $S_0$
according to a nonstationary distribution. We will denote by $\P_i$ the measure on joint sequences $(S_t,X_t)_{t \geq 0}$ given 
by fixing $S_0 = i$ and by $\P_{\mu}$ the measure given by choosing $S_0$ according to the distribution $\mu$:
\begin{align*}
\P_i(\cdot) = \P(\cdot | S_0 = i) ~\mbox{ and }~
\P_{\mu}( \cdot ) = \sum_i \mu_i \P_i(\cdot).
\end{align*} 
These measures $\P$, $\P_i$, and $\P_{\mu}$ are, of course, also extendable in a natural way to biinfinite sequences
$(S_t,X_t)_{t \in \Z}$ as oppose to one-sided sequences $(S_t,X_t)_{t \geq 0}$, and we will do so as necessary. 

\subsubsection{Equivalence of Model Types}
\label{sec:EquivalenceOfModelTypes}

Though they are indeed different objects state-emitting and edge-emitting HMMs are equivalent in the following sense:
Given an irreducible HMM $M$ of either type there exists an irreducible HMM $M'$ of the other type such that the stationary
output processes $(X_t)$ for the two HMMs $M$ and $M'$ are equal in distribution. We recall below the standard conversions. 

\begin{enumerate}
\item \emph{State-Emitting to Edge-Emitting} - If $M = (\SM, \XM, \TM, \OM)$ then $M' = (\SM, \XM, \{\TM^{'(x)}\})$, 
          where $\TM^{'(x)}_{ij} = \TM_{ij} \OM_{jx} $. 
\item \emph{Edge-Emitting to State-Emitting} - If $M = (\SM, \XM, \{ \TM^{(x)} \})$ then $M' = (\SM', \XM, \TM', \OM')$, 
          where $\SM' = \{(i,x) : \sum_j \TM^{(x)}_{ij} > 0 \}$, $\TM'_{(i,x) (j,y)} = 
          \left(\TM^{(x)}_{ij}/\sum_k  \TM^{(x)}_{ik}\right) \cdot \left(\sum_k \TM^{(y)}_{jk} \right)$, and $\OM'_{(i,x)y} = \indicator\{x=y\}$. 
\end{enumerate}

Note that the output $M'$ of the edge-emitting to state-emitting conversion is not just an arbitrary state-emitting HMM, but rather
is always a functional HMM. Thus, composition of the two conversion algorithms shows that functional HMMs are also 
equivalent to either state-emitting or edge-emitting HMMs. 

\subsubsection{Path-Mergeability}
\label{sec:PathMergeability}

For a HMM $M$, let $\delta_i(w)$ be the set of states $j$ that state $i$ can transition to upon emitting the word $w$:
\begin{align*}
\delta_i(w) & \equiv \{ j \in \SM: \P_i(X_0^{|w|-1} = w, S_{|w|} = j) > 0 \} ~,~ \mbox{for an edge-emitting HMM}. \\
\delta_i(w) & \equiv \{ j \in \SM: \P_i(X_1^{|w|} = w, S_{|w|} = j) > 0 \} ~,~ \mbox{for a state-emitting HMM}.
\end{align*}
In either case, if $w$ is the null word $\lambda$ then $\delta_i(w) \equiv \{i\}$, 
for each $i$. The following properties will be of central interest. 

\begin{Def}
A pair of states $i,j$ of a HMM $M$ is said to be \emph{path-mergeable} if there exists some word $w$ and state $k$ 
such that it is possible to transition from both $i$ and $j$ to $k$ on $w$. That is,  
\begin{align}
\label{eq:PMDef}
k \in \delta_i(w) \cap \delta_j(w). 
\end{align}
A HMM $M$ is said to have \emph{path-mergeable states}, or be \emph{path-mergeable}, if each pair of distinct 
states $i,j$ is path-mergeable.
\end{Def}

\begin{Def}
A symbol $x$ is said to be a \emph{flag} or \emph{flag symbol} for a state $k$ if it is possible to transition to state $k$ upon 
observing the symbol $x$, from any state $i$ for which it is possible to generate symbol $x$ as the next output. That is,
\begin{align}
\label{eq:FlagSymbolDef}
k \in \delta_i(x), \mbox{ for all $i$ with } \delta_i(x) \not= \emptyset. ~ \footnotemark{}
\end{align} 
A HMM $M$ is said to be \emph{flag-state} if each state $k$ has some flag symbol $x$.
\footnotetext{Note that if $x$ is a flag symbol for $k$, then after observing the symbol $x$ it is always possible for the HMM to 
be in state $k$ at the current time, regardless of its initial state or previous outputs. Thus, we call the symbol $x$ a flag
for the state $k$ as it signals or `flags' to the observer that $k$ is now possible as the current state of the HMM.}   
\end{Def}

Our end goal in Section \ref{sec:ResultsForEdgeEmittingHMMs}, below, is to prove exponential bounds on convergence 
of the entropy rate estimates $h(t)$ for edge-emitting HMMs with path-mergeable states. To do so, however, we will first  
prove similar bounds for edge-emitting HMMs under the flag-state hypothesis and then bootstrap. As we will show in Section 
\ref{sec:BlockModels}, if an edge-emitting HMM has path-mergeable states then the block model $M^n$, obtained by considering 
length-$n$ blocks of outputs as single output symbols, is flag-state, for some $n$. Thus, exponential convergence bounds for flag-state, 
edge-emitting HMMs pass to exponential bounds for path-mergeable, edge-emitting HMMs by considering block presentations. In Section 
\ref{sec:RelationToStateEmittingHMMs} we will also consider similar questions for state-emitting HMMs. In this case, analogous convergence
results follow easily from the results for edge-emitting HMMs by applying the standard state-emitting to edge-emitting conversion. 

\begin{Rem}
Note that if a state-emitting HMM has strictly positive transition probabilities in the underylying Markov chain, as considered in 
\cite{Birc62a, Hoch99a}, it is always path-mergeable. Similarly, if an edge-emitting or state-emitting HMM has strictly positive symbol 
emission probabilities and an aperiodic underlying Markov chain, as consider in \cite{Pfis03a}, then it is always path-mergeable. 
The converse of these statements, of course, do not hold. A concrete example will be given in the next subsection. 
\end{Rem}

\subsubsection{An Illustrative Example }
\label{sec:IllustrativeExample} 

\begin{figure}[h]
\begin{center} 
\includegraphics[scale=0.75]{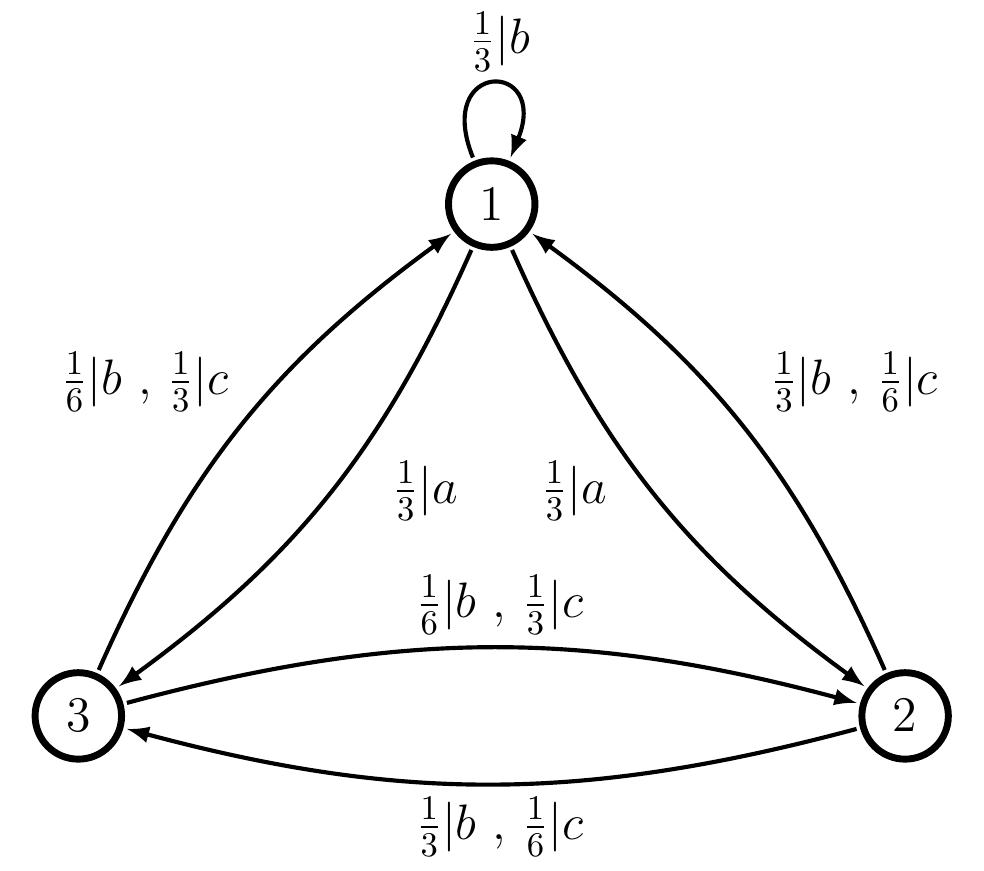}
\end{center}
\caption{Graphical depiction of the HMM $M$ described in Example \ref{ex:Example}. 
Edges are labeled $p|x$ for the transition probability $p = \TM_{ij}^{(x)}$ and symbol $x$. For visual clarity, parallel edges are omitted 
and each directed edge between states is labeled with all possible symbols upon which the transition between these two states may occur.} 
\label{fig:Example}
\end{figure}

To demonstrate the path-mergeability property and motivate its utility, we provide here a simple example of a 3-state, 3-symbol,
edge-emitting HMM $M$, which is path-mergeable, but such that neither $M$ or the equivalent functional HMM $M'$ given 
by the conversion algorithm of Section \ref{sec:EquivalenceOfModelTypes} satisfy any of the previously used conditions 
for establishing exponential convergence of the entropy rate estimates $h(t)$. Thus, although the entropy rate estimates 
for the output process of either model converge exponentially, this fact could not be deduced from previous results.

\begin{Exa}
\label{ex:Example}
Let $M$ be the edge-emitting HMM $(\SM, \XM, \{ \TM^{(x)} \})$ with $\SM = \{1,2,3\}$, $\XM = \{a,b,c\}$, and transition matrices
\begin{align*}
\TM^{(a)} = \left( \begin{array}{ccc} 0 & 1/3 & 1/3 \\ 0 & 0 & 0 \\ 0 & 0 & 0 \end{array} \right),~ 
\TM^{(b)} = \left( \begin{array}{ccc} 1/3 & 0 & 0 \\ 1/3 & 0 & 1/3 \\ 1/6 & 1/6 & 0 \end{array} \right),~
\TM^{(c)} = \left( \begin{array}{ccc} 0 & 0 & 0 \\ 1/6 & 0 & 1/6 \\ 1/3 & 1/3 & 0 \end{array} \right).
\end{align*} 
Also, let $M'$ be the equivalent (functional) state-emitting HMM constructed from $M$ by the conversion algorithm of 
Section \ref{sec:EquivalenceOfModelTypes}. For the readers convenience, a graphical depiction of the HMM 
$M$ is given in Figure \ref{fig:Example}. 

$M$ is, indeed, path-mergeable since each state $i$ can transition to state $1$ upon emitting the 1-symbol 
word $b$. However, $M$ clearly does not satisfy the unifilarity condition assumed in \cite{Trav11a, Trav11b} 
(or exactness condition used in \cite{Trav11a}), and neither $M$ or $M'$ satisfy any of the positivity conditions assumed
in \cite{Birc62a, Hoch99a, Pfis03a}. Moreover, $M'$ does not satisfy the conditions assumed in \cite{Han06} since its
output process does not have full support (the 2-symbol word $aa$ is forbidden). 
\end{Exa}

\subsubsection{Additional Notation}
\label{sec:AdditionalNotation}

The following additional notation and terminology for an edge emitting HMM $M = (\SM, \XM, \{ \TM^{(x)}\})$ 
will be used below for our proofs in Section \ref{sec:ResultsForEdgeEmittingHMMs}. 
\begin{itemize}

\item $\P(w)$ and $\P_i(w)$ denote, respectively, the probability 
of generating $w$ according to the measures $\P$ and $\P_i:$ 
\begin{align*}
\P(w) & \equiv \P(X_0^{|w|-1} = w) ~\mbox{ and }~ \P_i(w) \equiv \P_i(X_0^{|w|-1} = w)
\end{align*}
with the conventions $\P(\lambda) \equiv 1$ and $\P_i(\lambda) \equiv 1$, $i \in \SM$, for the null word $\lambda$.  

\item The \emph{process language} $\LM(M)$ for the HMM $M$ is the set of words $w$ of positive probability
in its stationary output process $(X_t)$, and $\LM_n(M)$ is the set of length-$n$ words in the process language. 
\begin{align*}
\LM(M) & \equiv \{w \in \XM^*: \P(w) > 0 \}. \\
\LM_n(M) & \equiv \{w \in \LM(M): |w| = n \}. 
\end{align*}

\item For $w \in \LM(M)$, $\SM(w)$ is the set of states that can generate $w$. 
\begin{align*}
\SM(w) \equiv \{i \in \SM: \P_i(w) > 0 \}.
\end{align*}

\item Finally, $\phi_i(w)$ is the distribution over the current state induced by the observing the output $w$ from initial state $i$, 
and $\phi(w)$ is the distribution over the current state induced by observing $w$ with the initial state chosen according to $\pi$. 
\begin{align*}
\phi_i(w) & \equiv \P_i(S_{|w|}|X_0^{|w|-1} = w). \\
\phi(w) & \equiv \P(S_{|w|}|X_0^{|w|-1} = w). 
\end{align*}
That is, $\phi_i(w)$ is the probability vector whose $k_{th}$ component is $\phi_i(w)_k = \P_i(S_{|w|} = k|X_0^{|w|-1} = w)$,  and
$\phi(w)$ is the probability vector whose $k_{th}$ component is $\phi(w)_k = \P(S_{|w|} = k|X_0^{|w|-1} = w)$. In the case $\P_i(w) = 0$ 
(respectively $\P(w) = 0$), $\phi_i(w)$ (respectively $\phi(w)$) is, by convention, defined to be the null distribution consisting of all zeros.
\end{itemize}
All above notation may also be used with time indexed symbol sequences $x_n^m$ in place of the word $w$ as well.
In this case, the lower time index $n$ is always ignored, and $x_n^m$ is treated simply as a length-$(m - n + 1)$ word. 
So, for example, $\P(x_n^m) = \P(X_0^{m-n} = x_n^m)$ and $\P_i(x_n^m) = \P_i(X_0^{m-n} = x_n^m)$.

\section{Results for Edge-Emitting HMMs}
\label{sec:ResultsForEdgeEmittingHMMs}

In this section we establish exponential convergence of the entropy rate block estimates $h(t)$ for edge-emitting HMMs
with path-mergeable states. The basic structure of the arguments is as follows:
\begin{enumerate}
\item We establish exponential bounds for flag-state (edge-emitting) HMMs.
\item We extend to path-mergeable (edge-emitting) HMMs by passing to a block model representation (see Section \ref{sec:BlockModels}). 
\end{enumerate}
Exponential convergence in the flag-state case is established by the following steps:
\begin{itemize}
\item[(i)] Using large deviation estimates on the reverse time generation space defined below in Section \ref{subsubsec:UpperBoundProbGtc},
we show that the set of ``good'' length-$t$ symbol sequences $G_t$ defined by (\ref{eq:DefGt}) has combined probability $1 - O(\mbox{exponentially small})$.
\item[(ii)] Using a coupling argument similar to that given in \cite{Birc62a} we show that $\norm{ \phi_k(x_0^{t-1}) - \phi_{\kh}(x_0^{t-1}) }_{TV}$
is exponentially small, for any symbol sequence $x_0^{t-1} \in G_t$ and states $k, {\kh} \in \SM(x_0^{t-1})$. 
\item[(iii)] Using (ii) we show that the difference $H(X_t|X_0^{t-1} = x_0^{t-1}) - H(X_t|X_0^{t-1} = x_0^{t-1}, S_0)$ is 
exponentially small for any $x_0^{t-1} \in G_t$. Combining this with the fact that $\P(G_t^c)$ is exponentially small shows that
the difference $H(X_t|X_0^{t-1}) - H(X_t|X_0^{t-1}, S_0)$ is exponentially small, from which exponential convergence
of the estimates $h(t)$ follows easily by a sandwiching argument. 
\end{itemize}

\subsection{Under Flag-State Assumption}
\label{sec:UnderFlagStateAssumption} 

Throughout Section \ref{sec:UnderFlagStateAssumption} we assume $M = (\SM, \XM, \{ \TM^{(x)} \})$ is a flag-state, 
edge-emitting HMM, and denote the flag symbol for state $j$ by $y_j$: $j \in \delta_i(y_j)$, for all $i$ with 
$\delta_i(y_j) \not= \emptyset$. Also, for notational convenience, we assume the state set is $\SM = \{1, ... , |\SM|\}$ and 
the output alphabet is $\XM = \{1, ... , |\XM| \}$. The constants $p_*, q_*, r_*, \eta, \alpha_1, \alpha_2$ for the HMM $M$ are defined as follows:
\begin{align}
\label{eq:DefsOfConstants}
p_j & \equiv \P(X_0 = y_j|S_1 = j) ~~ \mbox{ and } ~~ p_* \equiv \min_j p_j. \nonumber \\
q_j & \equiv \min_{i \in \SM(y_j)} \P_i(S_1 = j|X_0 = y_j) ~~ \mbox{ and } ~~ q_* \equiv \min_j q_j. \nonumber \\
r_* & \equiv \min_{i,j} \pi_i/\pi_j. \nonumber \\
\eta & \equiv \frac{p_* r_{*}}{2 |\SM|} ~,~
\alpha_1 \equiv \exp{\left(- \frac{p_*^2 r_{*}^2}{2 |\SM|^2}\right)} ~,~
\alpha_2 \equiv (1 - q_*^2)^{\eta}.
\end{align}
Note that, under the flag-state assumption, we always have $p_*, q_*, r_* \in (0,1]$, $\eta, \alpha_1 \in (0,1)$, and $\alpha_2 \in [0,1)$.
Our primary objective is to prove the following theorem:

\begin{The} 
\label{thm:ExponentialConvergenceEntropyRateEstimatesFlagState}
The entropy rate block estimates $h(t)$ for the HMM $M$ converge exponentially with
\begin{align*}
\limsup_{t \to \infty} \left\{ h(t) - h \right\}^{1/t} \leq \alpha
\end{align*}
where $\alpha \equiv \max\{\alpha_1, \alpha_2\}$.
\end{The}

The proof is, however, fairly lengthy. So, we have divided it into subsections following the steps (i)-(iii) outlined above. 
The motivation for the definitions of the various constants should become clear in the proof. 

\subsubsection{Upper Bound for $\P(G_t^c)$} 
\label{subsubsec:UpperBoundProbGtc}

For $x_0^{t-1} \in \LM_t(M)$, define
\begin{align*}
N(x_0^{t-1}) \equiv 
| \{0 \leq \tau \leq t-1 : x_{\tau} = y_k, \mbox{ for some } k \in \SM, \mbox{ and } \P_k(x_{\tau+1}^{t-1}) \geq \P_j(x_{\tau+1}^{t-1}), \forall j \in \SM \} | ~,
\end{align*}
and, for $t \in \N$, let 
\begin{align}
\label{eq:DefGt}
G_t \equiv \{ x_0^{t-1} \in \LM_t(M) : N(x_0^{t-1}) \geq \eta t \} ~\mbox{ and }~ G_t^c \equiv \LM_t(M)/G_t.
\end{align}
Recall that, according to our conventions, if $\tau = t -1$ then $x_{\tau+1}^{t-1} = x_t^{t-1}$ is the empty word $\lambda$,
and $\P_j(\lambda) = 1$, for each state $j \in \SM$. So, $N(x_0^{t-1})$ is indeed well defined. 

The purpose of this subsection is to prove the following lemma.

\begin{Lem}
\label{lem:BoundBelowProbGtc}
$\P(G_t^c) \leq \alpha_1^t$, for all $t \in \N$.
\end{Lem}

The proof of the lemma is based on large deviation estimates for an auxiliary sequence of random 
variables $(\Xt_t)_{t \in -\N}$, which is equal in law to the standard output sequence of the HMM $M$ on 
negative integer times, $(X_t)_{t \in -\N}$, but is defined on a separate explicit probability space $(\Omegat, \FMt, \Prt)$
(as opposed to the standard implicit probability space for our HMM $M$, $(\Omega, \FM, \P)$).  

For each $w \in \LM(M)$, fix a partition $P_w$ of the unit interval [0,1] into subintervals $I_{w,x}$, $x \in \XM$, such that:
\begin{enumerate}
\item The length of each $I_{w,x}$ is $\P(X_{-|w|-1} = x | X_{-|w|}^{-1} = w)$.
\item All length $0$ intervals $I_{w,x}$ are taken to be the empty set, rather than points. 
\item The leftmost interval $I_{w,x}$ is the closed interval $I_{w,y_{k(w)}} = [0, \P(X_{-|w|-1} = y_{k(w)} | X_{-|w|}^{-1} = w)]$ 
where $k(w) = \min \{k : \P_k(w) \geq \P_j(w), \forall j \in \SM \}$.
\end{enumerate} 
The \emph{reverse time generation space} $(\Omegat, \FMt, \Prt)$ and random variables $(\Xt_t)_{t \in -\N}$ on this space are defined as follows.
\begin{itemize}

\item $(\Ut_t)_{t \in -\N}$ is an i.i.d. sequence of uniform([0,1]) random variables.

\item $(\Omegat, \FMt, \Prt)$ is the canonical probability space (path space) on which the sequence $(\Ut_t)$ is defined 
(i.e. each point $\omega \in \Omegat$ is a sequence of real numbers $\omega = (u_t)_{t \in -\N}$ with $u_t \in [0,1]$, for all $t$).

\item On this space $(\Omegat, \FMt, \Prt)$, the random variables $\Xt_t, t \in - \N$, are defined inductively by:
\begin{enumerate}
\item $\Xt_{-1} = x$ if and only if $\Ut_{-1} \in I_{\lambda,x}$, (where $\lambda$ is the empty word).  
\item Conditioned on $\Xt_{t}^{-1} = w$ ($t \leq -1$), $\Xt_{t-1} = x$ if and only if $\Ut_{t-1} \in I_{w,x}$. 
\end{enumerate}
 
\end{itemize}
 
By induction on the length of $w$, it is easily seen that $\P(X_{-|w|}^{-1} = w) = \Prt(\Xt_{-|w|}^{-1} = w)$ 
for any word $w \in \XM^*$. So,
\begin{align}
\label{eq:EqualityInDistribution}
(\Xt_t)_{t \in -\N} \stackrel{d.}{=} (X_t)_{t \in - \N}.
\end{align}
Of course, this is a somewhat unnecessarily complicated way of constructing the output process of the HMM 
$M$ in reverse time. However, the explicit nature of the underlying space $(\Omegat, \FMt, \Prt)$ will be useful
in allowing us to translate large deviation estimates for the i.i.d. sequence $(\Ut_t)$ to large deviation 
estimates for the sequence $(\Xt_t)$. 

\begin{proof}[Proof of Lemma \ref{lem:BoundBelowProbGtc}]
If $w \in \LM(M)$ is any word with $\P_k(w) \geq \P_j(w)$, for all $j$, then by Bayes Theorem
\begin{align*}
\P(S_{-|w|} = k|X_{-|w|}^{-1} = w) \geq r_* \cdot \P(S_{-|w|} = j|X_{-|w|}^{-1} = w) ~,~ \mbox{ for all } j.
\end{align*}
Thus, for any such $w$, we have
\begin{align}
\label{eq:prob_statek_reverse_bound}
\P(  X&_{-|w|-1} = y_k | X_{-|w|}^{-1} = w) \nonumber \\
& \geq \P(S_{-|w|} = k|X_{-|w|}^{-1} = w) \cdot \P(X_{-|w|-1} = y_k| S_{-|w|} = k) \nonumber \\
& \geq \left( \frac{1}{|\SM|} \cdot \min_{j \in \SM} \left\{ \frac{ \P(S_{-|w|} = k|X_{-|w|}^{-1} = w)}{ \P(S_{-|w|} = j|X_{-|w|}^{-1} = w)} \right\} \right) \cdot \P(X_{-|w|-1} = y_k| S_{-|w|} = k) \nonumber \\
& \geq \frac{r_*}{|\SM|} \cdot p_* .
\end{align} 
The first inequality, of course, uses the fact that the HMM output sequence $X_t^{\infty}$ from time $t$ on is independent of 
the previous output $X_{-\infty}^{-t-1}$ conditioned on the state $S_t$.

Now, let $\Kt_0 \equiv 1$ and, for $t \in \N$, define the random variables $\Kt_t, \Nt_t, \Nt'_t$ on the reverse time generation space $(\Omegat, \FMt, \Prt)$ by
\begin{align*} 
\Kt_t & \equiv k(\Xt_{-t}^{-1}) = \min \{ k \in \SM : \P_k(\Xt_{-t}^{-1}) \geq \P_j(\Xt_{-t}^{-1}), \forall j \in \SM \}  ~, \\
\Nt_t & \equiv | \{ 0 \leq \tau \leq t-1 : \Xt_{-\tau-1} = y_{\Kt_\tau} \} | ~, \\
\Nt'_t & \equiv | \{ 0 \leq \tau \leq t-1 : \Ut_{-\tau-1} \leq p_* r_*/|\SM| \} |.
\end{align*}
By the estimate (\ref{eq:prob_statek_reverse_bound}) and the order of interval placement in $P_w$, we know that the random interval 
$I_{\Xt_{-\tau}^{-1}, y_{\Kt_{\tau}}}$  always contains $[0, p_* r_* / |\SM| ]$. So,
\begin{align*} 
\Ut_{-\tau-1} \leq p_* r_*/|\SM| ~\implies~ \Ut_{-\tau-1} \in I_{\Xt_{-\tau}^{-1}, y_{\Kt_{\tau}}} ~\implies~ \Xt_{-\tau-1} = y_{\Kt_{\tau}}
\end{align*} 
and, therefore, $\Nt_t$ is always lower bounded by $\Nt'_t$:
\begin{align*}
\Nt_t \geq \Nt'_t ~, \mbox{ for each } t \in \N. 
\end{align*} 
But, $\Nt'_t$ is simply $\sum_{\tau = 0}^{t-1} \Zt_{\tau}$, where $\Zt_{\tau} \equiv \indicator_{\{ \Ut_{-\tau-1} \leq p_* r_* / |\SM| \}}$. 
Since the $\Zt_{\tau}$s are i.i.d. with $\Prt(\Zt_{\tau} = 1) = p_* r_*/ |\SM|$, $\Prt(\Zt_{\tau} = 0) = 1 - p_* r_*/ |\SM|$ we may, 
therefore, apply Hoeffding's inequality to conclude that
\begin{align*}
\Prt(\Nt_t < \eta t) \leq \Prt(\Nt'_t < \eta t) \leq \alpha_1^t. 
\end{align*}
The lemma follows since $\Xt_{-t}^{-1} \stackrel{d.}{=} X_{-t}^{-1} \stackrel{d.}{=} X_{0}^{t-1}$, for each $t \in \N$,
which implies 
\begin{align*}
\P(G_t^c) \equiv \P \left( \{ x_0^{t-1} \in \LM_t(M) : N(x_0^{t-1}) < \eta t \} \right) \leq \Prt(\Nt_t < \eta t).
\end{align*}
\end{proof}

\subsubsection{Pair Chain Coupling Bound}
\label{subsubsec:PairChainCouplingBound}

In this subsection we prove the following lemma.

\begin{Lem}
\label{lem:CloseInducedDistributionsForSequenceInGt} 
For any $t \in \N$, $x_0^{t-1} \in G_t$, and states $k, {\kh} \in \SM(x_0^{t-1})$
\begin{align*}
\norm{ \phi_k(x_0^{t-1}) - \phi_{\kh}(x_0^{t-1})  }_{TV} \leq \alpha_2^t 
\end{align*} 
where $\norm{\mu - \nu}_{TV} \equiv \frac{1}{2} \norm{ \mu - \nu}_1$ is the total variational norm 
between two probability distributions (probability vectors) $\mu$ and $\nu$. 
\end{Lem}

The proof is based on a coupling argument for an auxiliary time-inhomogeneous Markov chain $(R_{\tau}, \Rh_{\tau})_{\tau = 0}^{t}$ 
on state space $|\SM| \times |\SM|$, which we call the \emph{pair chain}. This chain (defined for given $x_0^{t-1} \in G_t$ and states 
$k , \kh \in \SM(x_0^{t-1})$) is assumed to live on a separate probability space with measure $\Prh$, and is defined  by the following relations:
\begin{align*}
\Prh(R_0 = k, \Rh_0 = {\kh}) = 1 
\end{align*}
and 
\begin{align*}
\Prh(& R_{\tau+1} = j, \Rh_{\tau+1} = {\jh} | R_{\tau} = i, \Rh_{\tau} = {\ih}) \\
 & = \left\{ \begin{array}{lll}
  	\P(S_{\tau+1} = j|S_{\tau} = i, X_{\tau}^{t-1} = x_{\tau}^{t-1}) \cdot \P(S_{\tau+1} = {\jh}|S_{\tau} = {\ih}, X_{\tau}^{t-1} = x_{\tau}^{t-1}), \mbox{ if } i \not= \ih \\
  	\P(S_{\tau+1} = j|S_{\tau} = i, X_{\tau}^{t-1} = x_{\tau}^{t-1}), \mbox{ if } i = \ih \mbox{ and } j = \jh \\
	0, \mbox{ if } i = \ih \mbox{ and }  j \not= \jh\\
  \end{array} \right.
\end{align*}
for $0 \leq \tau \leq t-1$.

By marginalizing it follows that the state sequences $(R_{\tau})$ and $(\Rh_{\tau})$ are each individually
(time-inhomogeneuos) Markov chains with transition probabilities
\begin{align*}
\Prh(R_{\tau+1} &= j | R_{\tau} = i) = \P(S_{\tau+1} = j|S_{\tau} = i, X_{\tau}^{t-1} = x_{\tau}^{t-1}), \mbox{ and } \\
\Prh(\Rh_{\tau+1} &= {\jh} | \Rh_{\tau} = {\ih}) = \P(S_{\tau+1} = {\jh}|S_{\tau} = {\ih}, X_{\tau}^{t-1} = x_{\tau}^{t-1}).
\end{align*} 
Thus, for any $r_0^t, \rh_0^t \in \SM^{t+1}$, 
\begin{align*}
\Prh(R_0^{t} & = r_0^{t}) = \P(S_0^{t} = r_0^{t}|S_0 = k, X_0^{t-1} = x_0^{t-1}),  \mbox{ and }\\
\Prh(\Rh_0^{t} & = \rh_0^{t}) = \P(S_0^{t} =\rh_0^{t}|S_0 = {\kh}, X_0^{t-1} = x_0^{t-1}).
\end{align*}
So, we have the following coupling bound:
\begin{align}
\label{eq:PairChainCouplingBound}
\norm{ \phi_k(x_0^{t-1}) - \phi_{\kh}(x_0^{t-1}) }_{TV}
& \equiv \norm{\P_k(S_{t}|X_0^{t-1} = x_0^{t-1}) - \P_{\kh}(S_t|X_0^{t-1} = x_0^{t-1})}_{TV} \nonumber \\
& \leq \Prh(R_t \not= \Rh_t)
\end{align}
This bound will be used below to prove the lemma.

\begin{proof}[Proof of Lemma \ref{lem:CloseInducedDistributionsForSequenceInGt}] 
 
Fix $x_0^{t-1} \in G_t$ and states $k, \kh \in \SM(x_0^{t-1})$, and define the time set $\Gamma$ by
\begin{align*}
\Gamma \equiv \{0 \leq \tau \leq t-1 : x_{\tau} = y_{\ell}, \mbox{ for some } \ell \in \SM, \mbox{ and } \P_{\ell}(x_{\tau+1}^{t-1}) \geq \P_j(x_{\tau+1}^{t-1}), \forall j \in \SM \}.
\end{align*} 
Also, for $\tau \in \Gamma$, let $\ell_{\tau}$ be the state $\ell$ such that $x_{\tau} =  y_{\ell}$. 

Then, by Bayes Theorem, for any $\tau \in \Gamma$ and $i \in \SM(x_{\tau}^{t-1})$ we have
\begin{align*}
\P(S_{\tau+1} = \ell_{\tau}|S_{\tau} = i, X_{\tau}^{t-1} = x_{\tau}^{t-1}) \geq \P(S_{\tau+1} = \ell_{\tau} | S_{\tau} = i, X_{\tau} = y_{\ell_{\tau}}) \geq q_*.
\end{align*}
Combining this relation with the pair chain coupling bound (\ref{eq:PairChainCouplingBound}) gives
\begin{align*}
\norm{ \phi_k(x_0^{t-1}) - \phi_{\kh}(x_0^{t-1})  }_{TV}
& \leq \Prh(R_{t} \not= \Rh_{t}) \\
& = \prod_{\tau=0}^{t-1} \Prh\left(R_{\tau+1} \not= \Rh_{\tau+1} | R_{\tau} \not= \Rh_{\tau} \right) \\
& \leq \prod_{\tau \in \Gamma} \Prh\left(R_{\tau+1} \not= \Rh_{\tau+1} | R_{\tau} \not= \Rh_{\tau} \right) \\
& \leq  \prod_{\tau \in \Gamma} \max_{i, {\ih} \in \SM(x_{\tau}^{t-1}), i \not= \ih} ~ \Prh \left(R_{\tau+1} \not= \Rh_{\tau+1} | R_{\tau} = i, \Rh_{\tau} = {\ih} \right) \\
& \leq  \prod_{\tau \in \Gamma}  \max_{i, {\ih} \in \SM(x_{\tau}^{t-1}), i \not= \ih} ~ \left( 1 - \Prh\left( R_{\tau+1} = \Rh_{\tau+1} = \ell_{\tau} | R_{\tau} = i, \Rh_{\tau} = {\ih} \right)  \right) \\
& \leq \prod_{\tau \in \Gamma} (1 - q_*^2) \\
& \leq (1 - q_*^2)^{\eta t} \\
& = \alpha_2^t.
\end{align*}
The inequality in the second to last line follows from the fact that $x_0^{t-1} \in G_t$, which implies $|\Gamma| \geq \eta t$.
(Note: We have assumed in the proof that $\Prh(R_{\tau} \not= \Rh_{\tau}) > 0$, for all $0 \leq \tau \leq t-1$. If this is
not the case then $\Prh(R_{t} \not= \Rh_{t}) = 0$, so the conclusion follows trivially from (\ref{eq:PairChainCouplingBound}).) 
\end{proof}

\subsubsection{Convergence of Entropy Rate Approximations}
\label{ConvergenceOfEntropyRateApproximations}

In this subsection we will use the bounds given in the previous two subsections for $\P(G_t^c)$ and the difference in induced distributions 
$\norm{ \phi_k(x_0^{t-1}) - \phi_{\kh}(x_0^{t-1}) }_{TV}$, $x_0^{t-1} \in G_t$ and $k, \kh \in \SM(x_0^{t-1})$, to establish Theorem 
\ref{thm:ExponentialConvergenceEntropyRateEstimatesFlagState}. First, however, we will need two more simple lemmas. 

\begin{Lem}
\label{lem:nextsymbol_currentstate_TVbound}
Let $\mu$ and $\nu$ be two probability distributions on $\SM$, and let $\P_{\mu}(X_0)$ and $\P_{\nu}(X_0)$ denote, 
respectively, the probability distribution of the random variable $X_0$ with $S_0 \sim \mu$ and $S_0 \sim \nu$. Then 
\begin{align*}
\norm{\P_{\mu}(X_0) - \P_{\nu}(X_0)}_{TV} \leq \norm{\mu - \nu}_{TV}.
\end{align*}
\end{Lem}

\begin{proof}
It is equivalent to prove the statement for 1-norms. In this case, we have
\begin{align*}
\norm{\P_{\mu}(X_0) - \P_{\nu}(X_0)}_1 
& = \left\norm{\sum_k \mu_k \cdot \P_k(X_0) - \sum_k \nu_k \cdot \P_k(X_0)\right}_1 \\
& \leq \sum_k |\mu_k - \nu_k| \cdot \left\norm{\P_k(X_0)\right}_1 \\
& = \norm{\mu - \nu}_1
\end{align*}
(where $\P_k(X_0)$ is the distribution of the random variable $X_0$ when $S_0 = k$).
\end{proof}

\begin{Lem}
\label{lem:1norm_vs_entropy_diff}
Let $\mu = (\mu_1, ... , \mu_N)$ and $\nu = (\nu_1, ... , \nu_N)$ be two probability measures on 
the finite set $\{1, ... , N\}$. If $\norm{\mu - \nu}_{TV} \leq \epsilon$, with $\epsilon \in [0, 1/e]$, then
\begin{align*}
\left| H(\mu) - H(\nu) \right| \leq N \epsilon \log_2(1/\epsilon).
\end{align*}
\end{Lem}

\begin{proof}
If $\norm{\mu-\nu}_{TV} \leq \epsilon$ then $|\mu_k - \nu_k| \leq \epsilon$, for all $k$. Thus,
\begin{align*}
\left| H(\mu) - H(\nu) \right| 
& = \left| \sum_k \nu_k \log_2(\nu_k) - \mu_k \log_2(\mu_k)  \right| \\
& \leq \sum_k \left| \nu_k \log_2(\nu_k) - \mu_k \log_2(\mu_k)  \right| \\
& \leq N \cdot \max_{\epsilon' \in [0, \epsilon]} \max_{\xi \in [0,1 - \epsilon']}  |(\xi+\epsilon') \log_2(\xi+ \epsilon') - \xi \log_2(\xi)| \\
& = N \cdot \max_{\epsilon' \in [0, \epsilon]} \epsilon' \log_2(1/\epsilon') \\
& = N \cdot \epsilon \log_2(1/\epsilon).
\end{align*} 
The last two equalities, for $0 \leq \epsilon' \leq \epsilon \leq 1/e$, may be verified using single variable calculus techniques 
for maximization (recalling our conventions for continuous extensions of the functions $\xi \log(\xi)$ and $\xi \log(1/\xi)$ 
given in Section \ref{sec:TheEntropyRateAndFiniteBlockEstimates}). 
\end{proof}

\begin{proof}[Proof of Theorem \ref{thm:ExponentialConvergenceEntropyRateEstimatesFlagState}]
If $\alpha_2 = 0$, let $t_0 = 1$. Otherwise, let $t_0 = \lceil \log_{\alpha_2}(1/e) \rceil$. 
We claim, first, that for each $t \geq t_0$ and $x_0^{t-1} \in G_t$, 
\begin{align}
\label{eq:EntropyDiffInXtBoundOnGt}
H(X_t | X_0^{t-1} = x_0^{t-1}) - H(X_t|X_0^{t-1} = x_0^{t-1}, S_0) \leq |\XM| \cdot \alpha_2^t \cdot \log_2(1/\alpha_2^t).
\end{align} 
To see this, note that for any $k \in \SM(x_0^{t-1})$ Lemmas \ref{lem:CloseInducedDistributionsForSequenceInGt} and \ref{lem:nextsymbol_currentstate_TVbound} imply
\begin{align*}
\norm{ \P_k(X_t|&X_0^{t-1} = x_0^{t-1})  - \P(X_t|X_0^{t-1} = x_0^{t-1})  }_{TV} \\
\leq & ~\norm{ \P_k(S_t|X_0^{t-1} = x_0^{t-1}) - \P(S_t|X_0^{t-1} = x_0^{t-1})  }_{TV} \\
\leq & ~\max_{{\kh} \in \SM(x_0^{t-1})} \norm{ \P_k(S_t|X_0^{t-1} = x_0^{t-1}) - \P_{\kh}(S_t|X_0^{t-1} = x_0^{t-1})  }_{TV} \\
= & ~\max_{{\kh} \in \SM(x_0^{t-1})} \norm{ \phi_k(x_0^{t-1}) - \phi_{\kh}(x_0^{t-1}) }_{TV} \\ 
\leq & ~\alpha_2^t.
\end{align*} 
So, by Lemma \ref{lem:1norm_vs_entropy_diff},
\begin{align*}
H(X_t|X_0^{t-1} = x_0^{t-1}) - H(X_t|X_0^{t-1} = x_0^{t-1}, S_0 = k) \leq |\XM| \cdot \alpha_2^t \cdot \log_2(1/\alpha_2^t)
\end{align*} 
for each $k \in \SM(x_0^{t-1})$, as $t \geq t_0$. The claim (\ref{eq:EntropyDiffInXtBoundOnGt}) follows since the entropy $H(X_t|X_0^{t-1} = x_0^{t-1}, S_0)$ 
is, by definition, a weighted average of the entropies $H(X_t | X_0^{t-1} = x_0^{t-1}, S_0 = k)$, $k \in \SM(x_0^{t-1})$:
\begin{align*}
H(X_t|X_0^{t-1} = x_0^{t-1}, S_0) \equiv \hspace{- 3 mm} \sum_{k \in \SM(x_0^{t-1})} \P(S_0 = k | X_0^{t-1} = x_0^{t-1}) \cdot H(X_t | X_0^{t-1} = x_0^{t-1}, S_0 = k).
\end{align*}

Now, for any $t \in \N$,
\begin{align*}
h  & = \lim_{\tau \to \infty} H(X_0|X_{-\tau}^{-1})
\geq \lim_{\tau \to \infty} H(X_0|X_{-\tau}^{-1}, S_{-t})
= H(X_0|X_{-t}^{-1}, S_{-t})
= H(X_t|X_0^{t-1}, S_0).
\end{align*}
Thus, using Lemma \ref{lem:BoundBelowProbGtc} and the estimate (\ref{eq:EntropyDiffInXtBoundOnGt}) the difference
$h(t+1) - h$ may be bounded as follows for all $t \geq t_0$:
\begin{align*}
h(t+1) - h 
& = H(X_t|X_0^{t-1}) - h \\
& \leq H(X_t|X_0^{t-1}) - H(X_t|X_0^{t-1}, S_0) \\
& = \sum_{x_0^{t-1} \in \LM_t(M)} \P(x_0^{t-1}) \cdot \left[ H(X_t|X_0^{t-1} = x_0^{t-1}) - H(X_t|X_0^{t-1} = x_0^{t-1}, S_0) \right] \\
& \leq \P(G_t) \cdot \left( |\XM| \cdot \alpha_2^t \cdot \log_2(1/\alpha_2^t) \right) ~+~ \P(G_t^c) \cdot \log_2|\XM| \\
& \leq 1 \cdot  \left( |\XM| \cdot \alpha_2^t \cdot \log_2(1/\alpha_2^t) \right) ~+~ \alpha_1^t \cdot \log_2|\XM|.
\end{align*}
The theorem follows directly from this inequality. 
\end{proof} 

\subsection{Under Path-Mergeable Assumption}
\label{sec:UnderPathMergeableAssumption}

Building on the results of the previous section for flag-state HMMs, we now proceed to the proofs of 
exponential convergence of the entropy rate block estimates for path-mergeable HMMs. 
The general approach is as follows:
\begin{enumerate}
\item We show that for any path-mergeable HMM $M$ there is some $n \in \N$ such that the 
block model $M^n$ (defined below) is flag-state.
\item We combine Point 1 with the exponential convergence bound for flag-state HMMs given by
Theorem \ref{thm:ExponentialConvergenceEntropyRateEstimatesFlagState}
to obtain the desired bound for path-mergeable HMMs. 
\end{enumerate}
The main theorem, Theorem \ref{thm:ExponentialConvergenceEntropyRateEstimatesPathMergeable}, will be given 
in Section \ref{ConvergenceOfEntropyRateApproximations2} after introducing the block models in Section \ref{sec:BlockModels}. 

\subsubsection{Block Models} 
\label{sec:BlockModels} 

\begin{Def}
Let $M = (\SM, \XM, \{\TM^{(x)}\})$ be an edge-emitting hidden Markov model. For $n \in \N$, the \emph{block model} $M^n$ 
is the triple $(\SM, \WM, \{\QM^{(w)}\})$ where:
\begin{itemize}
\item $\WM = \LM_n(M)$ is the set of length-$n$ words of positive probability. 
\item $\QM_{ij}^{(w)} = \P_i(X_0^{n-1} = w, S_n = j)$ is the $n$-step transition probability from $i$ to $j$ on $w$. 
\end{itemize}
\end{Def}

One can show that if $M$ is irreducible and $n$ is relatively prime to the period of $M's$ graph, $per(M)$, then $M^n$ is also 
irreducible. Further, in this case $M$ and $M^n$ have the same stationary distribution $\pi$ and the stationary output process of
$M^n$ is the same (i.e. equal in distribution) to the stationary output process for $M$, when the latter is considered over length-$n$
blocks rather than individual symbols. That is, for any $w_0^{t-1} \in \WM^t$,
\begin{align*}
\P(X_0^{tn-1} = w_0^{t-1}) = \P^n(W_0^{t-1} = w_0^{t-1})
\end{align*}
where $W_t$ denotes the $t_{th}$ output of the block model $M^n$, and $\P^n$ is the probability distribution over the output 
sequence $(W_t)$ when the initial state of the block model is chosen according to the stationary distribution $\pi$. 

The following important lemma allows us to reduce questions for path-mergeable HMMs to analogous 
questions for flag-state HMMs by considering such block presentations. 

\begin{Lem}
\label{lem:BlockModelIsFlagState}
If $M = (\SM, \XM, \{\TM^{(x)}\}) $ is an edge-emitting HMM with path-mergeable states, then there exists some $n \in \N$,
relatively prime to $per(M)$, such that the block model $M^n= (\SM, \WM, \{\QM^{(w)}\})$ is flag-state. 
\end{Lem}

\begin{proof} 
Extending the flag symbol definition (\ref{eq:FlagSymbolDef}) to multi-symbol words, we will say a word $w \in \LM(M)$ is 
a \emph{flag word} for state $k$ if all states that can generate $w$ can transition to $k$ on $w$:
\begin{align*}
k \in \delta_i(w) ~, \mbox{ for all } i \mbox{ with } \P_i(w) > 0.  
\end{align*}
The following two facts are immediate from this definition:
\begin{enumerate}
\item If $w$ is a flag word for state $i$ and $v$ is any word with $j \in \delta_i(v)$, then $wv \in \LM(M)$ is a flag word for state $j$.
\item If $w$ is a flag word for state $j$ and $v$ is any word with $vw \in \LM(M)$, then $vw$ is also a flag word for $j$.
\end{enumerate}

Irreducibility of the HMM $M$ along with Fact 1 ensure that if some state $i \in \SM$ has a flag word then each state $j \in \SM$ 
has a flag word.  Moreover, by Fact 2, we know that in this case the flag words may all be chosen of some fixed length $n$,
relatively prime to $per(M)$, which implies the block model $M^n$ is flag-state. Thus, it suffices to show that there is a single 
state $i$ with a flag word. 

Below we will explicitly construct a word $v^*$, which is a flag word for some state $i^*$. For notational convenience, we will 
assume throughout that the state set is $\SM = \{1,2,...,|\SM|\}$ and the alphabet is $\XM = \{1,2,...,|\XM| \}$.  Also, for $i,j \in \SM$ 
we will denote by $w_{ij}$ and ${k_{ij}}$ the special word $w$ and state $k$ satisfying the path-mergeability condition (\ref{eq:PMDef}), 
so that $k_{ij} \in \delta_i(w_{ij}) \cap \delta_j(w_{ij})$.

$v^*$ is then constructed by the following algorithm: 

\begin{itemize}
\item $t := 0$, $v_0 := \lambda$ (the empty word), $i_0 := 1$, $\RM := \SM/\{1\}$ 
\item While $\RM \not= \emptyset$ do:
\begin{enumerate} 
\item[] $k_t :=  \min \{k: k \in \RM \}$
\item[] $j_t := \min \{j: j \in \delta_{k_t}(v_t)\}$
\item[] $w_t := w_{i_t j_t}$ 
\item[] $v_{t+1} := v_t w_t$
\item[] $i_{t+1} := k_{i_t j_t}$
\item[] $\RM := \RM / \left(  \{k : i_{t+1} \in \delta_k(v_{t+1}) \} \cup \{k : \P_k(v_{t+1}) = 0 \} \right)$ 
\item[] $t := t+1$ 
\end{enumerate}
\item $v^* := v_t, i^* := i_t$
\end{itemize}

By construction we always have $i_{t+1} \in \delta_{k_t}(v_{t+1}) = \delta_{k_t}(v_t w_t)$. Thus, 
the state $k_t$ is removed from the set $\RM$ in each iteration of the loop, so the algorithm 
must terminate after a finite number of steps. Now, let us assume that the loop terminates at time $t$ with word 
$v^* = v_t = w_0 w_1 ... w_{t-1}$ and state $i^* = i_t$. Then, since $i_0 = 1$ and $i_{\tau+1} \in \delta_{i_{\tau}}(w_{\tau})$ 
for each $\tau = 0,1,...,t-1$, we know $i^* \in \delta_1(v^*)$, and, hence, that $v^* \in \LM(M)$. 
Further, since each state $k \not= 1$ must be removed from the list $\RM$ before the algorithm
terminates, we know that for each $k \not=1$ one of the following must hold:
\begin{enumerate}
\item $i_{\tau+1} \in \delta_k(v_{\tau+1})$, for some $0 \leq \tau \leq t-1$, which implies $i^* \in \delta_k(v^*)$.
\item $\P_k(v_{\tau+1}) = 0$, for some $0 \leq \tau \leq t-1$, which implies $\P_k(v^*) = 0$. 
\end{enumerate}
It follows that $v^*$ is a flag word for state $i^*$. 
\end{proof}

\subsubsection{Convergence of Entropy Rate Approximations}
\label{ConvergenceOfEntropyRateApproximations2}
 
\begin{The}
\label{thm:ExponentialConvergenceEntropyRateEstimatesPathMergeable} 
Let $M$ be an edge-emitting HMM, and let $h$ and $h(t)$ denote the entropy rate and length-$t$ entropy rate estimate for its output process. 
If the block model $M^n$ is flag-state, for some $n$ relatively prime to $per(M)$, then the estimates $h(t)$ converge exponentially with
\begin{align}
\label{eq:ExponentialConvergenceEntropyRateEstimatesPathMergeable}
\limsup_{t \to \infty} \left\{h(t) - h \right\}^{1/t} \leq \alpha^{1/n} 
\end{align}
where $0 < \alpha < 1$ is the convergence rate, given by Theorem \ref{thm:ExponentialConvergenceEntropyRateEstimatesFlagState},
for the entropy estimates of the block model. In particular, by Lemma \ref{lem:BlockModelIsFlagState}, the entropy estimates always
convergence exponentially if $M$ is path-mergeable. 
\end{The} 
 
\begin{Rems} ~
\begin{enumerate}

\item If $M$ is path-megreable then Lemma \ref{lem:BlockModelIsFlagState} ensures that there will always be some $n$, relatively prime to $per(M)$,
such that the block model $M^n$ is flag-state. However, to obtain an ideal bound on the speed of convergence of the entropy rate estimates for $M$,
one should generally choose $n$ as small as possible, and the construction given in the lemma is not always ideal for this. 

\item The theorem does not explicitly require the HMM $M$ to be path-mergeable for exponential convergence to hold, only that the block model $M^n$ is flag-state, 
for some $n$. However, while it is easy to check if a HMM $M$ is path-mergeable or flag-state (see Section \ref{sec:Testability} for the path-mergeability test), 
it is generally much more difficult to determine if the block model $M^n$ is flag-state for some unknown $n$. Indeed, the computational time to do so may be 
super-exponential in the number of states. And, while asymptotically ``almost all'' HMMs are path-mergeable in the sense of Propostion \ref{prop:Typicality} 
below, ``almost none'' are themselves flag-state. Thus, the flag-state condition in practice is not as directly applicable as the path-mergeability condition.  

\item Another somewhat more useful condition is incompatibility. We say states $i$ and $j$ are \emph{incompatible} if there exists some length $m$
such that the sets of allowed words of length $m$\footnotemark{} that can be generated from states $i$ and $j$ have no overlap (i.e., for
each $w \in \LM_m(M)$, either $\P_i(w) = 0$, or $\P_j(w) = 0$, or both). \footnotetext{Hence, also, the sets of allowed words of all lengths $m' > m.$}

By a small modification of the construction used in Lemma \ref{lem:BlockModelIsFlagState}, it is easily seen that if each pair of distinct states $i,j$ 
of a HMM $M$ is either path-mergeable or incompatible then the block model $M^n$ will be flag-state, for some $n$. This weaker sufficient 
condition may sometimes be useful in practice since incompatibility of state pairs, like path-mergeability, may be checked in reasonable computational
time (using a similar algorithm). 

\end{enumerate}
\end{Rems}
 
\begin{proof}
The claim (\ref{eq:ExponentialConvergenceEntropyRateEstimatesPathMergeable}) is an immediate consequence of 
Theorem \ref{thm:ExponentialConvergenceEntropyRateEstimatesFlagState} and the following simple lemma. 
\end{proof}

\begin{Lem}
\label{lem:EntropyRateConvergenceBlockScaling}
Let $M$ be an edge-emitting HMM with block model $M^n$, for some $n$ relatively prime to $per(M)$. Let $h$ and $h(t)$ be the entropy rate and length-$t$ block
estimate for the output process of $M$, and let $g$ and $g(t)$ be the entropy rate and length-$t$ block estimate for the output process of $M^n$. Then:
\begin{itemize}
\item[(i)] $g = nh$.
\item[(ii)] $\limsup_{t \to \infty} \{h(t) - h\}^{1/t} = \left[ \limsup_{t \to \infty} \{g(t) - g\}^{1/t} \right]^{1/n}$.
\end{itemize}
\end{Lem} 

\begin{proof}
(i) is a direct computation:
\begin{align*}
g = \lim_{t \to \infty} \frac{H(W_0^{t-1})}{t} = \lim_{t \to \infty} \frac{H(X_0^{nt-1})}{t} = \lim_{t \to \infty} n \cdot \frac{H(X_0^{nt-1})}{nt} = nh
\end{align*}
where $X_t$ and $W_t$ denote, respectively, the $t_{th}$ output symbol of $M$ and $M^n$. 
To show (ii), note that by the entropy chain rule (see, e.g., \cite{Cove06a})
\begin{align*}
g(t+1) - g 
= H(W_t|W_0^{t-1}) - nh
= H(X_{nt}^{n(t+1)-1} | X_0^{nt-1}) - nh \\
= \left\{ \sum_{\tau = nt}^{n(t+1)-1} H(X_{\tau}|X_0^{\tau-1}) \right\} - nh 
= \sum_{\tau = nt}^{n(t+1)-1} \{h(\tau+1) - h\}.
\end{align*} 
The claim follows from this relation and the fact that the block estimates $h(t)$ are weakly decreasing (i.e. nonincreasing). 
\end{proof} 

\section{Relation to State-Emitting HMMs} 
\label{sec:RelationToStateEmittingHMMs}

In Section \ref{sec:ResultsForEdgeEmittingHMMs} we established exponential convergence of the entropy rate estimates
$h(t)$ for path-mergeable, edge-emitting HMMs. The following simple proposition shows that path-mergeability, as well as 
the flag-state property, both translate directly from state-emitting to edge-emitting HMMs. Thus, exponential convergence 
of the entropy rate estimates $h(t)$ also holds for path-mergeable, state-emitting HMMs.

\begin{Prop}
\label{prop:EquivalencePropsStateEdgeEmitting}
Let $M = (\SM, \XM, \TM, \OM)$ be a state-emitting HMM, and let $M' = (\SM, \XM, \{\TM^{'(x)}\})$ be the corresponding
edge-emitting HMM defined by the conversion algorithm of Section \ref{sec:EquivalenceOfModelTypes} with 
$\TM^{'(x)}_{ij} = \TM_{ij} \OM_{jx}$. Then:
\begin{enumerate}
\item[(i)] $M$ is path-mergeable if and only if $M'$ is path-mergeable.
\item[(ii)] $M$ is a flag-state HMM if and only if $M'$ is a flag-state HMM. 
\end{enumerate}
\end{Prop}

\begin{proof}
By induction on $t$, it is easily seen that for each $t \in \N$, $x_1^t \in \XM^t$, and $s_0^t \in \SM^{t+1}$
\begin{align*}
\P_i(S_0^t = s_0^t, X_1^t = x_1^t) = \P_i'(S_0^t = s_0^t, X_0^{t-1} = x_1^t) 
\end{align*}
where $\P_i$ and $\P_i'$ are, respectively, the measures on state-symbol sequences $(S_t, X_t)_{t \geq 0}$ for $M$ and $M'$
from initial state $S_0 = i$. This implies that the transition function $\delta_i(w)$ for $M$ is the same as the transition function 
$\delta'_i(w)$ for $M'$, i.e. $\delta_i(w) = \delta'_i(w)$ for each $i \in \SM, w \in \XM^*$. Both claims follow immediately from 
the equivalence of transition functions. 
\end{proof}

\section{Typicality and Testability of the Path-Meregeability Property} 
\label{sec:TypicalityTestabilityOfPathMergeability}

\subsection{Testability}
\label{sec:Testability}

The following algorithm to test for path-mergeability of state pairs in an edge-emitting HMM is a small modification of the
standard table filling algorithm for deterministic finite automata (DFAs) given in \cite{Hopc01}. To test for path-mergeability
in a state-emitting HMM one may first convert it to an edge-emitting HMM and then test the edge-emitting HMM for path-mergeability, 
as the result will be equivalent by Proposition \ref{prop:EquivalencePropsStateEdgeEmitting}.

\begin{Alg}
\label{alg:PathMergeabilityTest}
Test for path-mergeable state pairs in an edge-emitting HMM. 
\begin{enumerate}
\item \emph{Initialization Step}
\begin{itemize} 
\item Create a table with boxes for each pair of distinct states $(i,j)$. Initially all boxes are unmarked.
\item Then, for each state pair $(i,j)$, mark the box for pair $(i,j)$ if there is some symbol $x$ with $\delta_i(x) \cap \delta_j(x) \not= \emptyset$.
\end{itemize}
\item \emph{Inductive Step} 
\begin{itemize} 
\item If the box for pair $(i',j')$ is already marked and $i' \in \delta_i(x), j' \in \delta_j(x)$, for some symbol $x$, then mark the box for pair $(i,j)$. 
\item Repeat until no more new boxes can be marked this way.
\end{itemize}
\end{enumerate}
\end{Alg}

By induction on the length of the minimum path-merging word $w$ for a given state pair $(i,j)$ it is easily seen that 
a state pair $(i,j)$ ends up with a marked box under Algorithm \ref{alg:PathMergeabilityTest} if and only if it is path-mergeable.
Thus, the HMM is itself path-mergeable if and only if all state pairs $(i,j)$ end up with marked boxes. This algorithm is also
reasonably fast (polynomial time) if the inductions in the second step are carried out in an efficient manner.
In particular, a decent encoding of the inductive step gives a maximum run time $O(m \cdot n^4)$ \footnotemark{}.

\footnotetext{For the standard DFA table filling algorithm an efficient encoding of the inductive step, as described in \cite{Hopc01}, 
gives a maximal run time $O(m \cdot n^2)$. However, DFAs, by definition, have deterministic transitions; for each state $i$ 
and symbol $x$ there is only 1 outgoing edge from state $i$ labeled with symbol $x$. The increased run time for checking path-mergeability
is due to the fact that each state $i$ may have up to $n$ outgoing transitions on any symbol $x$ in a general HMM topology.} 

\subsection{Typicality}
\label{sec:Typicality}

By an \emph{edge-emitting HMM topology} we mean simply the directed, symbol-labeled graph, without assigned probabilities 
for the allowed transitions. Similarly, by a \emph{state-emitting HMM topology} we mean the directed graph of allowed state transitions 
along with the sets of allowed output symbols for each state, without the probabilities of allowed transitions or emission probabilities of 
allowed symbols from each state. By a \emph{labeled, $n$-state, $m$-symbol topology} (either edge-emtting or state-emitting) we mean 
a topology with $n$ states $\{1,...,n\}$ and $m$ symbols $\{1,...,m\}$, or some subset thereof. That is, we allow that not all the symbols 
are actually used. 

Clearly, the path-mergeability property depends only on the topology of a HMM. The following proposition shows that this property is 
asymptotically typical in the space of HMM topologies. 

\begin{Prop}
\label{prop:Typicality}
If a HMM topology is selected uniformly at random from all irreducible, labeled, $n$-state, $m$-symbol topologies then it will be path-mergeable 
with probability $1 - O(\alpha^n)$, uniformly in $m$,  for some fixed constant $0 < \alpha < 1$. 
\end{Prop}

\begin{Rems} ~
\begin{enumerate}
\item The claim holds for either a randomly selected edge-emitting topology or a randomly selected state-emitting topology, though we will present 
the proof only in the edge-emitting case as the other case is quite similar.
\item A similar statement also holds if one does not require the topology to be irreducible, and, in fact, the proof is somewhat simpler in this case. 
However, because we are interested only in irreducible HMMs where there is a unique stationary state distribution $\pi$ and well defined 
entropy rate, we feel it is more appropriate to consider a randomly selected irreducible HMM topology. 
\item For fixed $m \in \N$, there are exponentially more $n$-state, $(m+1)$-symbol topologies than $n$-state, $m$-symbol topologies,
as $n \rightarrow \infty$. Thus, it follows from the proposition that a similar claim holds if one considers topologies with $n$-states and exactly
$m$ symbols used, rather than the $n$-state, $m$-symbol topologies we consider in which some of the symbols may not be used.
\item We are considering labeled topologies, so that the state labels $1,...,n$ and symbol labels $1,...,m$ are not interchangeable. 
That is, two topologies that would be the same under a permutation of the symbol and/or state labels are considered distinct, not equivalent. 
We believe a similar statement should also hold if one considers a topology-permutation-equivalence class chosen uniformly at random from 
all such equivalence classes with $n$ states and $m$ symbols. However, the proof in this case seems more difficult. 
\end{enumerate}
\end{Rems}

\begin{proof}[Proof (edge-emitting)]
Assume, at first, that a topology $T$ is selected uniformly at random from all labeled, $n$-state, $m$-symbol topologies without requiring 
irreducibility, or even that each state has any outgoing edges. That is, for each pair of states $i,j$ and symbol $x$ we have, independently, 
a directed edge from state $i$ to state $j$ labeled with symbol $x$ present with probability 1/2 and absent with probability 1/2. 

Let $A_{i,x} \equiv \{j : \exists \mbox{ an edge from $i$ to $j$ labeled with $x$}\}$ be the set of states that state $i$ can transition to on symbol 
$x$, and let $N_{i,x} \equiv |A_{i,x}|$ be the number of states that state $i$ can transition to on symbol $x$. Also, define the following events:
\begin{align*}
& E_{1,x} \equiv \{ N_{i,x} > n/4, \mbox{ for all } i \} ~,~ x \in \XM \\
& E_{2,x} \equiv \{ A_{i,x} \cap A_{j,x} \not= \emptyset, \mbox{ for all } i \not= j\} ~,~ x \in \XM \\
& E_3 \equiv \{ i \rightarrow j, \mbox{ for all } i,j \} = \{T \mbox{ is irreducible} \}
\end{align*}
where $i \rightarrow j$ means there is a path from $i$ to $j$ in the topology (directed graph) $T$. We note that: 
\begin{enumerate}

\item By Hoeffding's inequality, 
\begin{align*}
\P(N_{i,x} \leq n/4) = \P(N_{i,x} - \E(N_{i,x}) \leq - n/4) \leq e^{-2n(1/4)^2} = e^{-n/8}.
\end{align*}
So, for each $x$, $\P(E_{1,x}^c) \leq n e^{-n/8}$.

\item Since the sets $A_{i,x}, A_{j,x}$ are independent for all $i \not= j$, $\P( A_{i,x} \cap A_{j,x} = \emptyset |E_{1,x}) \leq (3/4)^{n/4}$, 
for all $i \not= j$.  Hence, $\P(E_{2,x}^c |E_{1,x}) \leq {n \choose 2} (3/4)^{n/4}$, for each $x$.

\item Since there are at least $\lceil n/4 \rceil -1$ states in the set $A_{i,x}/\{i\}$ on the event $E_{1,x}$, we have
\begin{align*}
\P(i \not\rightarrow j |E_{1,x}) 
\leq \P(j \not\in A_{k,x}, \mbox{ for all } k \in A_{i,x}/\{i\} | E_{1,x})
\leq (3/4)^{n/4-1}.
\end{align*}
Hence, $\P(E_3^c | E_{1,x}) \leq n^2 (3/4)^{n/4-1}$, for each $x$. 

\end{enumerate}
Now, the probability that an irreducible, labeled, $n$-state, $m$-symbol HMM topology selected uniformly at random is 
path-mergeable is $\P(T \mbox{ is path-mergeable} | E_3)$, where $\P$ is, as above, the probability measure on the random 
topologies $T$ given by selecting each directed symbol-labeled edge to be, independently, present or absent with probability 
1/2. Using points 1,2, and 3, and fixing any symbol $x \in \XM$, we may bound this probabilities follows: 
\begin{align*}
\P(T & \mbox{ is path-mergeable} | E_3) \\
& \geq \P(T \mbox{ is path-mergeable}, E_3) \\
& \geq \P(T \mbox{ is path-mergeable}, E_3, E_{1,x}) \\
& = \P(E_{1,x}) \cdot \P(T \mbox{ is path-mergeable}, E_3 |E_{1,x}) \\
& = \P(E_{1,x}) \cdot \left(1 - \P(T \mbox{ is not path-mergeable}, E_3|E_{1,x}) - \P(E_3^c|E_{1,x}) \right) \\
& \geq \P(E_{1,x}) \cdot \left(1 - \P(E_{2,x}^c| E_{1,x}) - \P(E_3^c|E_{1,x}) \right) \\
&\geq (1 - ne^{-n/8}) \cdot \left(1 - {n \choose 2}(3/4)^{n/4} - n^2 (3/4)^{n/4-1} \right) \\
& = 1 - O(\alpha^n) ~,
\end{align*} 
for any $\alpha > (3/4)^{1/4}$. (Note that $(3/4)^{1/4} \approx 0.931 > e^{-1/8} \approx 0.882$.)

\end{proof}

\section*{Acknowledgments}
The author thanks Jim Crutchfield for helpful discussions. This work was partially supported by ARO grant W911NF-12-1-0234 and VIGRE grant DMS0636297. 

\bibliography{ref,chaos}

\end{document}

%% file: shortcuts.tex
\theoremstyle{plain}   \newtheorem{Alg}{Algorithm}
\theoremstyle{plain}   \newtheorem{Lem}{Lemma}
\theoremstyle{plain} 	
\theoremstyle{plain} 	\newtheorem{The}{Theorem}
\theoremstyle{plain} 	\newtheorem{Prop}{Proposition}
\theoremstyle{plain} 	
\theoremstyle{plain}	\newtheorem*{Rem}{Remark}
\theoremstyle{plain}	\newtheorem*{Rems}{Remarks}
\theoremstyle{plain}	\newtheorem{Def}{Definition} 
\theoremstyle{plain}	
\theoremstyle{plain}   
\theoremstyle{plain} \newtheorem{Exa}{Example}
\theoremstyle{plain} 

\def\norm#1{\|#1\|}

\def\implies{\Longrightarrow}

\def\indicator{\mathds{1}}

\def\E{\mathbb{E}}

\def\N{\mathbb{N}}

\def\P{\mathbb{P}}

\def\Z{\mathbb{Z}}

\def\FM{\mathcal{F}}

\def\LM{\mathcal{L}}

\def\OM{\mathcal{O}}

\def\QM{\mathcal{Q}}
\def\RM{\mathcal{R}}
\def\SM{\mathcal{S}}
\def\TM{\mathcal{T}}

\def\WM{\mathcal{W}}
\def\XM{\mathcal{X}}
\def\YM{\mathcal{Y}}

\def\Kt{\widetilde{K}}

\def\Nt{\widetilde{N}}

\def\Ut{\widetilde{U}}

\def\Xt{\widetilde{X}}

\def\Zt{\widetilde{Z}}

\def\Rh{\widehat{R}}

\def\ih{\widehat{i}}
\def\jh{\widehat{j}}
\def\kh{\widehat{k}}

\def\rh{\widehat{r}}

\def\Prt{\mathbb{\widetilde{P}}}
\def\Prh{\mathbb{\widehat{P}}}

\def\FMt{\widetilde{\FM}}
\def\Omegat{\widetilde{\Omega}}

%% file: pm.bbl
\begin{thebibliography}{10}

\bibitem{Gilb59a}
E.~J. Gilbert.
\newblock On the identifiability problem for functions of finite {Markov}
  chains.
\newblock {\em Ann. Math. Statist.}, 30(3):688--697, 1959.

\bibitem{Blac57a}
D.~Blackwell and L.~Koopmans.
\newblock On the identifiability problem for functions of finite {Markov}
  chains.
\newblock {\em Ann. Math. Statist.}, 28(4):1011--1015, 1957.

\bibitem{Blac57b}
D.~Blackwell.
\newblock The entropy of functions of finite-state {Markov} chains.
\newblock In {\em Transactions of the first Prague conference on information
  theory, statistical decision functions, random processes}, pages 13--20.
  Publishing House of the Czechoslovak Academy of Sciences, 1957.

\bibitem{Juan91a}
B.~H. Juang and L.~R. Rabiner.
\newblock Hidden {Markov} models for speech recognition.
\newblock {\em Technometrics}, 33(3):251--272, 1991.

\bibitem{Rabi89a}
L.~R. Rabiner.
\newblock A tutorial on hidden {Markov} models and selected applications in
  speech recognition.
\newblock {\em IEEE Proc.}, 77:257--286, 1989.

\bibitem{Bahl83a}
L.~R. Bahl, F.~Jelinek, and R.~L. Mercer.
\newblock A maximum likelihood approach to continuous speech recognition.
\newblock {\em IEEE Trans. Pattern Analysis and Machine Intelligence},
  PAMI-5:179--190, 1983.

\bibitem{Jeli76a}
F.~Jelinek.
\newblock Continuous speech recognition by statistical methods.
\newblock {\em Proc. IEEE}, 64:532--536, 1976.

\bibitem{Siep04a}
A.~Siepel and D.~Haussler.
\newblock Combining phylogenetic and hidden {Markov} models in biosequence
  analysis.
\newblock {\em J. Comp. Bio.}, 11(2-3):413--428, 2004.

\bibitem{Eddy98a}
S.~R. Eddy.
\newblock Profile hidden {Markov} models.
\newblock {\em Bioinformatics}, 14(9):755--763, 1998.

\bibitem{Karp98a}
K.~Karplus, C.~Barrett, and R.~Hughey.
\newblock Hidden {Markov} models for detecting remote protein homologies.
\newblock {\em Bioinformatics}, 14(10):846--856, 1998.

\bibitem{Eddy95a}
S.~R. Eddy, G.~Mitchison, and R.~Durbin.
\newblock Maximum discrimination hidden {Markov} models of sequence consensus.
\newblock {\em J. Comp. Bio.}, 2(1):9--23, 1995.

\bibitem{Bald94a}
P.~Baldi, Y.~Chauvin, T.~Hunkapiller, and M.~A. McClure.
\newblock Hidden {Markov} models of biological primary sequence information.
\newblock {\em PNAS}, 91:1059--1063, 1994.

\bibitem{Crut01a}
J.~P. Crutchfield and D.~P. Feldman.
\newblock Regularities unseen, randomness observed: Levels of entropy
  convergence.
\newblock {\em CHAOS}, 13(1):25--54, 2003.

\bibitem{Sowe92a}
R.~B. Sowers and A.~M. Makowski.
\newblock Discrete-time filtering for linear systems in correlated noise with
  non-gaussian initial conditions: Formulas and asymptotics.
\newblock {\em IEEE Trans. Automat. Control}, 37:114--121, 1992.

\bibitem{Atar95a}
R.~Atar and O.~Zeitouni.
\newblock Lyapunov exponents for finite state nonlinear filtering.
\newblock {\em SIAM J. Control Optim.}, 35:36--55, 1995.

\bibitem{Atar97a}
R.~Atar and O.~Zeitouni.
\newblock Exponential stability for nonlinear filtering.
\newblock {\em Ann. Inst. H. Poincare Prob. Statist.}, 33(6):697--725, 1997.

\bibitem{Glan00b}
F.~Le Gland and L.~Mevel.
\newblock Exponential forgetting and geometric ergodicity in hidden {Markov}
  models.
\newblock {\em Math. Control Signals Systems}, 13:63--93, 2000.

\bibitem{Chig04a}
P.~Chigansky and R.~Lipster.
\newblock Stability of nonlinear filters in nonmixing case.
\newblock {\em Ann. App. Prob.}, 14(4):2038--2056, 2004.

\bibitem{Douc09a}
R.~Douc, G.~Fort, E.~Moulines, and P.~Priouret.
\newblock Forgetting the initial distribution for hidden {Markov} models.
\newblock {\em Stoch. Proc. App.}, 119(4):1235--1256, 2009.

\bibitem{Coll09a}
P.~Collet and F.~Leonardi.
\newblock Loss of memory of hidden {Markov} models and {Lyapunov} exponents.
\newblock {\em arXiv/0908.0077}, 2009.

\bibitem{Birc62a}
J.~Birch.
\newblock Approximations for the entropy for functions of {Markov} chains.
\newblock {\em Ann. Math. Statist}, 33(3):930--938, 1962.

\bibitem{Hoch99a}
B.M. Hochwald and P.R. Jelenkovic.
\newblock State learning and mixing in entropy of hidden {Markov} processes and
  the {Gilbert-Elliott} channel.
\newblock {\em IEEE Trans. Info. Theory}, 45(1):128--138, 1999.

\bibitem{Pfis03a}
H.~Pfister.
\newblock {\em On the capacity of finite state channels and analysis of
  convolutional accumulate-m codes}.
\newblock PhD thesis, University of California, San Diego, 2003.

\bibitem{Trav11a}
N.~F. Travers and J.~P. Crutchfield.
\newblock Exact synchronization for finite-state sources.
\newblock {\em J. Stat. Phys.}, 145(5):1181--1201, 2011.

\bibitem{Trav11b}
N.~F. Travers and J.~P. Crutchfield.
\newblock Asymptotic synchronization for finite-state sources.
\newblock {\em J. Stat. Phys.}, 145(5):1202--1223, 2011.

\bibitem{Han06}
G.~Han and B.~Marcus.
\newblock Analyticity of entropy rate of hidden {Markov} chains.
\newblock {\em IEEE Trans. Info. Theory}, 52(12):5251--5266, 2006.

\bibitem{Cove06a}
T.~M. Cover and J.~A. Thomas.
\newblock {\em Elements of Information Theory}.
\newblock Wiley-Interscience, Hoboken, NJ, second edition, 2006.

\bibitem{Hopc01}
J.~E. Hopcroft, R.~Motwani, and J.~D. Ullman.
\newblock {\em Introduction to Automata Theory, Languages, and Computation}.
\newblock Addison-Wesley, second edition, 2001.

\end{thebibliography}
